\documentclass[a4paper,11pt]{amsart}
\usepackage{amsmath,amsthm,amssymb}
\usepackage{amsfonts, amscd}
\usepackage{hyperref}

\theoremstyle{plain}
\newtheorem{theorem}{Theorem}[section]
\newtheorem{proposition}[theorem]{Proposition}
\newtheorem{lemma}[theorem]{Lemma}

\newtheorem{corollary}[theorem]{Corollary}

\theoremstyle{plain} \numberwithin{equation}{section}
\theoremstyle{definition}

\newtheorem{remark}[theorem]{Remark}

\newtheorem{example}[theorem]{Example}

\newcommand{\Z}{\mathbb{Z}}
\newcommand{\N}{\mathbb{N}}
\newcommand{\Q}{\mathbb{Q}}

\newcommand{\C}{\mathbb{C}}
\newcommand{\CP}{\mathbb{C}P}

\newcommand{\R}{\mathbb{R}}
\newcommand{\uC}{\underline{\mathbb{C}}}

\newcommand{\cH}{\mathcal{H}}

\DeclareMathOperator{\Aut}{Aut}

\begin{document}%
\title[$\Q$-trivial Bott manifolds]{Classification of $\Q$-trivial Bott manifolds}

\author[S. Choi]{Suyoung Choi}
\address{Department of Mathematics, Osaka City University, Sugimoto, Sumiyoshi-ku, Osaka 558-8585, Japan}
\email{choi@sci.osaka-cu.ac.jp}

\urladdr{http://math01.sci.osaka-cu.ac.jp/~choi}
\author[M. Masuda]{Mikiya Masuda}
\address{Department of Mathematics, Osaka City University, Sugimoto, Sumiyoshi-ku, Osaka 558-8585, Japan}
\email{masuda@sci.osaka-cu.ac.jp}
\thanks{The first author was supported by the Japanese Society for the Promotion of Sciences (JSPS grant no. P09023)
and the second author was partially supported by Grant-in-Aid for Scientific Research 19204007.}


\begin{abstract}
A Bott manifold is a closed smooth manifold obtained as the total space of an iterated $\C P^1$-bundle starting with a point, where each $\C P^1$-bundle is the projectivization of a Whitney sum of two complex line bundles.  A \emph{$\Q$-trivial Bott manifold} of dimension $2n$ is a Bott manifold whose cohomology ring is isomorphic to that of $(\CP^1)^n$ with $\Q$-coefficients. We find all diffeomorphism types of $\Q$-trivial Bott manifolds and show that they are distinguished by their cohomology rings with $\Z$-coefficients. As a consequence, we see that the number of diffeomorphism classes in $\Q$-trivial Bott manifolds of dimension $2n$ is equal to the number of partitions of $n$.  We even show that any cohomology ring isomorphism between two $\Q$-trivial Bott manifolds is induced by a diffeomorphism.
\end{abstract}

\date{\today}
\maketitle

\section{Introduction}
A Bott tower of height $n$ is a sequence of $\CP^1$-bundles
\begin{equation} \label{eqn:Bott tower}
B_n\stackrel{\pi_n}\longrightarrow B_{n-1} \stackrel{\pi_{n-1}}\longrightarrow
\dots \stackrel{\pi_2}\longrightarrow B_1 \stackrel{\pi_1}\longrightarrow
B_0=\{\text{a point}\},
\end{equation}
where each $\pi_i\colon B_i\to B_{i-1}$ for $i=1,\dots,n$ is the projectivization of a Whitney sum of two complex line bundles over $B_{i-1}$. We call $B_i$ an \emph{$i$-stage Bott manifold} and are concerned with the diffeomorphism type of the $n$-stage Bott manifold $B_n$.  Note that even if two Bott towers of height $n$ are different, their $n$-stage Bott manifolds can be diffeomorphic.

If the fiber bundles in \eqref{eqn:Bott tower} are all trivial, then $B_n$ is diffeomorphic to $(\CP^1)^n$.  It is shown in \cite{Ma-Pa-2008} that if the cohomology ring of $B_n$ is isomorphic to that of $(\CP^1)^n$ with $\Z$-coefficients as graded rings, then $B_n$ is diffeomorphic to $(\C P^1)^n$ and moreover the fiber bundles in \eqref{eqn:Bott tower} are all trivial.

We say that $B_n$ is \emph{$\Q$-trivial} if its cohomology ring is isomorphic to that of $(\CP^1)^n$ with $\Q$-coefficients as graded rings.
In this paper, we shall find all diffeomorphism types of $\Q$-trivial Bott manifolds and show that they are diffeomorphic if and only if their cohomology rings with $\Z$-coefficients are isomorphic as graded rings (Theorem~\ref{theorem:Q-trivial Bott manifold}). As a consequence, we see that the number of diffeomorphism classes in $\Q$-trivial Bott manifolds of dimension $2n$ is equal to the number of partitions of $n$.  We also prove that any automorphism of the cohomology ring of a $\Q$-trivial Bott manifold is induced by a diffeomorphism.  This implies that any cohomology ring isomorphism between two $\Q$-trivial Bott manifolds is induced by a diffeomorphism since we already establish that the diffeomorphism types of $\Q$-trivial Bott manifolds are distinguished by their cohomology rings.

Our study is motivated by the so-called \emph{cohomological rigidity problem} for toric manifolds. A toric manifold is a non-singular compact complex algebraic variety with an algebraic torus action having a dense orbit. The cohomological rigidity problem for toric manifolds asks whether the topological types of toric manifolds are distinguished by their cohomology rings or not (see \cite{Ma-Su-2008}). This problem is open, but we have some affirmative partial solutions to the problem for (generalized) Bott manifolds in \cite{Ma-Pa-2008}, \cite{Ch-Ma-Su-2010}, \cite{ch-su-pre} and \cite{Ch-Par-Su-pre}. The result of this paper provides another affirmative evidence to the problem for Bott manifolds.  One can consider the real analogue of Bott towers and Bott manifolds, but the cohomological rigidity for \emph{real} Bott manifolds is established with $\Z/2$-coefficients, see \cite{Ka-Ma-2009} and \cite{Masuda-2008}.

This paper is organized as follows. In Section \ref{section:Bott manifolds}, we review Bott manifolds and prepare several lemmas to prove our main theorems. We find all diffeomorphism types of $\Q$-trivial Bott manifolds in Section \ref{section : Q-trivial Bott manifold} and prove the cohomological rigidity for $\Q$-trivial Bott manifolds in Section \ref{section: cohomological rigidity of Q-trivial}.  Section~\ref{sect:auto} is devoted to proving that any automorphism of the cohomology ring of a $\Q$-trivial Bott manifold is induced by a diffeomorphism.

Throughout this paper, cohomology is taken with $\Z$-coefficient unless otherwise stated.

\section{Cohomology of Bott manifolds} \label{section:Bott manifolds}

We begin with recalling some general facts on projective bundles.  Let $\pi\colon E\to B$ be a complex vector bundle over a smooth manifold $B$ and let $P(E)$ be the projectivization of $E$.

\begin{lemma}\cite[Lemma 2.1]{Ch-Ma-Su-2010} \label{lemm:line}
Let $B$ and $E$ be as above and let $L$ be a complex line bundle over $B$. We denote by $E^*$ the complex vector bundle dual to $E$. Then both $P(E^*)$ and $P(E\otimes L)$ are isomorphic to $P(E)$ as fiber bundles over $B$, in particular, they are diffeomorphic.
\end{lemma}

\begin{proof}
We shall reproduce the proof given in \cite{Ch-Ma-Su-2010} for the reader's convenience sake.

Choose a Hermitian metric $\langle\ ,\ \rangle$ on $E$, which is anti-$\C$-linear on the first entry and $\C$-linear on the second entry, and define a map $\tilde b\colon E\to E^*$ by $\tilde b(u):=\langle u,\ \rangle$.  This map is not $\C$-linear but anti-$\C$-linear, so it induces a map $b\colon P(E)\to P(E^*)$, which gives an isomorphism as fiber bundles.

For each $x\in B$, we choose a non-zero vector $v_x$ from the fiber of $L$ over $x$ and define a map $\tilde c\colon E\to E\otimes L$ by $\tilde c(u_x):=u_x\otimes v_x$ where $u_x$ is an element of the fiber of $E$ over $x$. The map $\tilde c$ depends on the choice of $v_x$'s but the induced map $c\colon P(E)\to P(E\otimes L)$ does not because $L$ is a line bundle. It is easy to check that $c$ gives an isomorphism of $P(E)$ and $P(E\otimes L)$ as fiber bundles over $B$.
\end{proof}

\begin{remark} \label{rema:line}
The bundle map $b\colon P(E)\to P(E^*)$ does not preserve the canonical complex structures on the fibers and the pullback of the tautological line bundle over $P(E^*)$ by $b$ is complex conjugate to the tautological line bundle over $P(E)$ since $\tilde b$ is anti-$\C$-linear.  On the other hand, the bundle map $c\colon P(E)\to P(E\otimes L)$ above preserves the canonical complex structure on the fibers and pulls back the tautological line bundle over $P(E\otimes L)$ to that over $P(E)$.
\end{remark}

If $H^{odd}(B)=0$ (and this is the case for Bott manifolds), then $H^*(P(E))$ is a free module over $H^*(B)$ via $\pi^*\colon H^*(B)\to H^*(P(E))$ and the Borel-Hirzebruch formula \cite[(2) on p.515]{bo-hi58} tells us that
\begin{equation} \label{eqn:BH}
    H^\ast(P(E)) = H^\ast(B)[x]/\big(\sum_{i=0}^m (-1)^ic_{i}(E)x^{m-i}\big),
\end{equation}
where $m$ is the fiber dimension of $E$, $c_i(E)$ denotes the $i$-th Chern class of $E$, and $x$ denotes the first Chern class of the tautological line bundle over $P(E)$.  Moreover, the tangent bundle $T_fP(E)$ along the fibers of $P(E)\to B$ admits a canonical complex structure since each fiber is a complex projective space, and with this complex structure its total Chern class is given by
\begin{equation} \label{eqn:tf}
  c(T_fP(E))=\sum_{i=0}^m(1-x)^{m-i}c_i(E).
\end{equation}

Now we consider the Bott tower \eqref{eqn:Bott tower}.  Each fiber bundle $\pi_j\colon B_j\to B_{j-1}$ for $j=1,\dots,n$ is the projectivization of a Whitney sum of two complex line bundles by definition and we may assume that one of the two line bundles is trivial by Lemma~\ref{lemm:line}. Therefore, one can express
\[
\text{$B_j=P(\uC\oplus\gamma^{\alpha_j})$ with $\alpha_j\in H^2(B_{j-1})$,}
\]
where $\uC$ denotes the trivial complex line bundle and $\gamma^{\alpha_j}$ denotes the complex line bundle over $B_{j-1}$ with $\alpha_j$ as the first Chern class.  Note that $\alpha_1=0$ since $B_0$ is a point.  Let $x_j$ be the first Chern class of the tautological line bundle over $B_j$.  Then it follows from \eqref{eqn:BH} that
$$
    H^\ast(B_j) = H^\ast(B_{j-1})[x_j]/ \big(x_j^2 = \alpha_j x_j\big).
$$
Using this formula inductively on $j$ and regarding $H^*(B_j)$ as a graded subring of $H^*(B_n)$ through the projections in \eqref{eqn:Bott tower}, we see that
\begin{equation} \label{eqn:HBn}
    H^*(B_n)=\Z[x_1, \ldots, x_n]/\big(x_j^2 = \alpha_jx_j\mid j=1,\dots,n\big).
\end{equation}

Sometimes it is convenient and helpful to express
\[
\alpha_j=\sum_{i=1}^{j-1} A^i_j x_i \quad\text{with $A^i_j\in \Z$}
\]
and form an upper triangular matrix of size $n$ with zero diagonals:
$$
    A=\left(
      \begin{array}{ccccc}
        0 & A^1_2 & A^1_3 &\cdots & A^1_n \\
         & 0 & A^2_3 & \cdots & A^2_n \\
         &  & \ddots & \ddots& \vdots \\
         &   & &0 &A^{n-1}_n\\
         &   & & &0
      \end{array}
    \right).
$$

Let $S^1$ and $S^3$ denote the unit sphere of $\C$ and $\C^2$ respectively.  Using the matrix $A$, one can describe $B_n$ as the quotient of $(S^3)^n$ by a free action of $(S^1)^n$ defined by
\begin{equation} \label{eqn:quotient}
\begin{split}
&(t_1,\dots,t_n)\cdot \big((z_1,w_1),\dots,(z_j,w_j),\dots,(z_n,w_n)\big)\\
=&\big((t_1z_1,t_1w_1),\dots,(t_jz_j,(\prod_{i=1}^{j-1}t_i^{-A^i_j})t_jw_j),\dots,(t_nz_n,(\prod_{i=1}^{n-1}t_i^{-A^i_n})t_nw_n)\big)
\end{split}
\end{equation}
where $(t_1,\dots,t_n)\in (S^1)^n$ and $(z_j,w_j)$ denotes the coordinate of the $j$th component of $(S^3)^n$. In fact, the projections $$(S^3)^n\to (S^3)^{n-1}\to\dots \to S^3\to \text{\{a point\}}$$ defined by dropping the last factor at each stage induces the Bott tower \eqref{eqn:Bott tower}.

The next lemma and corollary are tricks to simplify algebraic computations. An ordered pair $(z, \bar{z})$ of elements in $H^2(B_n)$ is said to be \emph{vanishing} if $z\bar{z} = 0$ and \emph{primitive} if both $z$ and $\bar{z}$ are primitive.  Note that $(x_j,x_j-\alpha_j)$ is a primitive vanishing pair for each $j$ since $x_j^2=\alpha_jx_j$.

\begin{lemma}\label{lemma:masuda's tricky}
A primitive vanishing pair $(z, \bar{z})$ is of the form
$$
    (ax_j+ u, \pm (a(x_j - \alpha_j) - u))
$$ for some $j$, where $a$ is a non-zero integer, $u$ is a linear combination of $x_i$'s with $i<j$, and $u(u+a \alpha_j) = 0$.
\end{lemma}
\begin{proof}
Set $z= ax_j + u$ (resp. $\bar{z} = bx_k + v$), where $a$ (resp. $b$) is a non-zero integer and $u$ (resp. $v$) is a linear combination of $x_i$'s with $i<j$ (resp. $i<k$). If $k \neq j$, then $ab x_j x_k$ term in $z\bar{z}$ survives in $H^\ast(B_n)$ because of \eqref{eqn:HBn}, hence $k=j$. Therefore,
\begin{equation} \label{eqn:zz}
    0=z\bar{z}  = abx_j^2 + (av + bu)x_j + uv  = (ab \alpha_j + av + bu)x_j + uv.
\end{equation}
Since $u$ and $v$ are linear combinations of $x_i$'s with $i<j$, the identity \eqref{eqn:zz} implies that
\begin{equation} \label{eqn:vanish}
\text{$ab \alpha_j + av +bu =0$\quad  and\quad  $uv = 0$.}
\end{equation}
The former identity in \eqref{eqn:vanish} shows that $bu$ is divisible by $a$. However $u$ is not divisible by any nontrivial factor of $a$ since $z=ax_j+u$ is primitive. Hence $a | b$. Similarly, $av$ is divisible by $b$ and hence $b | a$. Therefore, $b = \pm a$ and hence $v = \mp (u +a\alpha_j)$ by the former identity of \eqref{eqn:vanish}.  This proves the first statement in the lemma because $\bar z=bz_j+v$.  The last identity in the lemma follows from the latter identity of \eqref{eqn:vanish} since $v=u+a\alpha_j$ up to sign.
\end{proof}

\begin{corollary}\label{corollary: squared zero elements}
A square zero primitive element in $H^2(B_n)$ is either $x_j - \frac{1}{2}\alpha_j$ or $2x_j - \alpha_j$ up to sign for some $j$, where $\alpha_j^2 = 0$ in both cases. In particular, the number of square zero primitive elements in $H^2(B_n)$ up to sign is equal to the number of $\alpha_j$'s with $\alpha_j^2 = 0$.
\end{corollary}
\begin{proof}
Since $z = \bar{z}$ in the proof of Lemma~\ref{lemma:masuda's tricky}, either $2u = -a \alpha_j$ or $2x_j = \alpha_j$. But the latter case does not occur since $\alpha_j$ is a linear combination of $x_i$'s with $i<j$. Hence, $2u = -a\alpha_j$. Thus, it follows from the primitiveness of $z$ that $z$ must be either $x_j-\frac{1}{2}\alpha_j$ or $2x_j - \alpha_j$ up to sign. Since $u(u+a\alpha_j)=0$ and $2u=-a\alpha_j$, we have $\alpha_j^2=0$, proving the corollary.
\end{proof}

\section{$\Q$-trivial Bott manifolds} \label{section : Q-trivial Bott manifold}

The purpose of this section is to classify $\Q$-trivial Bott manifolds. We freely use the notation in Section~\ref{section:Bott manifolds}.

\begin{proposition}\label{proposition:Q-trivial Bott manifold}
$B_n$ is $\Q$-trivial if and only if $\alpha_j^2 = 0$ in $H^*(B_n)$ for all $j=1, \ldots, n$.
In particular, if $B_n$ is $\Q$-trivial, then every Bott manifold $B_j$ in the tower \eqref{eqn:Bott tower} is $\Q$-trivial.
\end{proposition}
\begin{proof}
If $\alpha_j^2 =0$, then $( x_j - \frac{\alpha_j}{2})^2=0$ in $H^\ast(B_n; \Q)$ because $x_j^2=\alpha_j x_j$.  Since $ x_j - \frac{\alpha_j}{2}$ for $j=1,\dots,n$ generate $H^*(B_n;\Q)$ as a graded ring, this shows that $B_n$ is $\Q$-trivial. Conversely, if $B_n$ is $\Q$-trivial, there are $n$ primitive elements in $H^2(B_n)$ up to sign whose square vanish. By Corollary \ref{corollary: squared zero elements}, the number of $\alpha_j$'s whose square vanish is also $n$, which implies the converse.
\end{proof}

\begin{example} \label{example:Hirzebruch surface}
For $a\in \Z$, let $\Sigma_a = P(\uC \oplus \gamma^{ax_1})$, where $\gamma^{ax_1}$ is the complex line bundle over $\CP^1=B_1$ whose first Chern class is $ax_1\in H^2(\C P^1)$.  $\Sigma_a$ is called a \emph{Hirzebruch Surface}, which was first studied by Hirzebruch in \cite{Hirzebruch-1951}. Note that
$$
    H^\ast(\Sigma_a;\Z) = \Z[x_1, x_2]/ (x_1^2=0,\ x_2^2 =ax_1),
$$
so that $\alpha_1 = 0$ and $\alpha_2 = ax_1$ in this case. Since the squares of $\alpha_1$ and $\alpha_2$ are both $0$, $\Sigma_a$ is $\Q$-trivial. As is well-known, $\Sigma_a$ is diffeomorphic to $\CP^1 \times \CP^1$ if $a$ is even and to $\CP^2 \sharp \overline{\CP^2}$ if $a$ is odd.
\end{example}

Denote $\cH_1 = \CP^1, \cH_2 = \Sigma_1$ and let $\pi_2 : \cH_2 \to \cH_1$ be the canonical projection. We consider the pullback bundle $\pi_3 : \cH_3 \to \cH_2$ of $\pi_2:\cH_2 \to \cH_1$ via $\pi_2$;
\begin{equation} \label{eqn:H32}
\begin{CD}
    \cH_3 @>\rho_3>>  \cH_2 = P(\uC \oplus \gamma^{x_1}) \\
@VV{\pi_3}V  @VV{\pi_2}V\\
 \cH_2 = P(\uC \oplus \gamma^{x_1}) @>{\pi_2}>> \cH_1 = \CP^1
 \end{CD}
\end{equation}
where $\rho_3$ denotes the induced bundle map.
Then $\cH_3$ is a $3$-stage Bott manifold, in fact, $\cH_3 = P(\uC \oplus \gamma^{x_1})$ where $\uC$ and $\gamma^{x_1}$ are both regarded as complex line bundles over $\cH_2$.  Therefore, the matrix corresponding to the Bott tower
$$\cH_3\stackrel{\pi_3}\longrightarrow \cH_2\stackrel{\pi_2}\longrightarrow \cH_1\stackrel{\pi_1}\longrightarrow \{\text{a point}\}$$
is given by
$$
    \left(
      \begin{array}{ccc}
        0 & 1 & 1 \\
          & 0 & 0 \\
          &   & 0 \\
      \end{array}
    \right).
$$
Since the pullback of the tautological line bundle over $\cH_2$ by $\rho_3$ in \eqref{eqn:H32} is the tautological line bundle over $\cH_3$, we have $\rho_3^*(x_2)=x_3$, while $\rho_3^*(x_1)=x_1$ which follows from the commutativity of the diagram \eqref{eqn:H32}.

Inductively, we shall define $\cH_n$ as follows:
\begin{equation} \label{eqn:suyoung_tower}
\begin{CD}
    \cH_n @>\rho_n>> \cH_{n-1} @>\rho_{n-1}>> \dots @>\rho_4>> \cH_3 @>\rho_3>> \cH_2\\
    @VV{\pi_n}V  @VV{\pi_{n-1}}V   @. @VV{\pi_3}V @VV{\pi_2}V\\
    \cH_{n-1} @>{\pi_{n-1}}>> \cH_{n-2} @>{\pi_{n-2}}>> \dots @>\pi_3>> \cH_2 @>{\pi_2}>> \cH_1.
\end{CD}
\end{equation}
Note that
\begin{equation} \label{eqn:rhon}
\cH_n\stackrel{\pi_n}\longrightarrow \cH_{n-1}\stackrel{\pi_{n-1}}\longrightarrow \dots \stackrel{\pi_2}\longrightarrow \cH_1
\stackrel{\pi_1}\longrightarrow \{\text{a point}\}
\end{equation}
is a Bott tower of height $n$ corresponding to the $n\times n$-matrix
\begin{equation} \label{eqn:matrixHn}
    \left(
      \begin{array}{ccccc}
        0 & 1 & 1 &\cdots  & 1 \\
          & 0 & 0 & \cdots & 0 \\
          &   & 0 &  \cdots & 0  \\
           &  &  & \ddots & \vdots \\
         &  &  &  & 0 \\
      \end{array}
    \right)
\end{equation}
and
\begin{equation} \label{eqn:Hn}
    H^\ast(\cH_n) = \Z[x_1, \ldots, x_n]/ (x_1^2 =0,\ x_j^2 =x_1x_j \text{ for }j=2, \ldots, n),
\end{equation}
so that $\alpha_1=0$ and $\alpha_j = x_1$ for all $j=2, \ldots, n$. Since $\alpha_j^2 = 0$ for any $j$, $\cH_n$ is a $\Q$-trivial Bott manifold by Proposition \ref{proposition:Q-trivial Bott manifold}.  We also note that $\rho_j\colon \cH_j\to \cH_{j-1}$ $(j>2)$ is a bundle map and pulls back the tautological line bundle over $\cH_{j-1}$ to that of $\cH_j$, so that
\begin{equation} \label{eqn:rhoj}
\begin{split}
&\rho_j^*(x_{j-1})=x_j \quad\text{for $j>2$, while}\\
&\rho_j^*(x_1)=x_1\quad \text{by the commutativity of \eqref{eqn:suyoung_tower}.}
\end{split}
\end{equation}

\begin{lemma} \label{lemm:mod2}
Square zero primitive elements in $H^2(\cH_n)$ are
\[
\text{$\pm x_1$ and $\pm(2x_j-x_1)$ for $j>1$.}
\]
In particular, their mod 2 reductions are equal to the mod 2 reduction of $x_1$.
\end{lemma}

\begin{proof}
Since $\alpha_1=0$ and $\alpha_j=x_1$ for $j>1$ in \eqref{eqn:Hn}, the lemma is an immediate consequence of Corollary~\ref{corollary: squared zero elements}.
\end{proof}

Note that the mod 2 reduction of a square zero element of $H^2(\cH_n)$ is either zero or equal to the mod 2 reduction of $x_1$ by Lemma~\ref{lemm:mod2}.

\begin{lemma} \label{lemma:bundle over H_n}
If $\alpha$ is a square zero element in $H^2(\cH_n)$, then
\[
P(\uC\oplus\gamma^\alpha)\cong
\begin{cases}
P(\uC\oplus\uC)=\cH_n\times \cH_1 \quad &\text{if $\alpha=0$ in $H^2(\cH_n)\otimes\Z/2$,}\\
P(\uC\oplus\gamma^{x_1})=\cH_{n+1} \quad&\text{if $\alpha=x_1$ in $H^2(\cH_n)\otimes\Z/2$,}
\end{cases}
\]
as bundles over $\cH_n$.
\end{lemma}
\begin{proof}
By Lemma~\ref{lemm:mod2}, $\alpha$ is either $ax_1$ or $a(2x_j-x_1)$ for $j>1$, where $a$ is an integer.  Thus it suffices to prove
\begin{enumerate}
  \item \label{item:1} $P(\gamma^{ax_1}\oplus \uC)\cong P(\gamma^{(a+2b)x_1} \oplus \uC)$ as bundles for any $b \in \Z$,
  \item \label{item:2} $P(\gamma^{a( 2x_j-x_1)} \oplus \uC)\cong P(\gamma^{-ax_1} \oplus \uC)$ as bundles for any $j>1$.
\end{enumerate}

We first prove (1).  By Lemma~\ref{lemm:line} we have
\begin{equation*}
P(\gamma^{ax_1}\oplus \uC)\cong P((\gamma^{ax_1}\oplus \uC)\otimes \gamma^{bx_1})=P(\gamma^{(a+b)x_1}\oplus\gamma^{bx_1})\quad\text{as bundles}.
\end{equation*}
Therefore it suffices to prove
\begin{equation} \label{eqn:(1)}
P(\gamma^{(a+b)x_1}\oplus\gamma^{bx_1})\cong P(\gamma^{(a+2b)x_1}\oplus\uC)\quad\text{as bundles}.
\end{equation}
All line bundles involved in \eqref{eqn:(1)} are the pullback of line bundles over $\cH_1$ by a composition of the projections $\pi_i$'s in the tower \eqref{eqn:rhon}.  Therefore it suffices to prove \eqref{eqn:(1)} when the base space is $\cH_1$.  But then the two vector bundles $\gamma^{(a+b)x_1}\oplus\gamma^{bx_1}$ and $\gamma^{(a+2b)x_1}\oplus\uC$ in \eqref{eqn:(1)} are isomorphic because their total Chern classes are same and complex vector bundles over $\cH_1=\CP^1$ are classified by their total Chern classes as is well-known.

The proof of (2) is similar to that of (1).
By Lemma~\ref{lemm:line} we have
\begin{equation*}
P(\gamma^{a(2x_j-x_1)}\oplus \uC)\cong P((\gamma^{a(2x_j-x_1)}\oplus \uC)\otimes \gamma^{-ax_j})=P(\gamma^{a(x_j-x_1)}\oplus\gamma^{-ax_j}).
\end{equation*}
Therefore it suffices to prove
\begin{equation} \label{eqn:(2)}
P(\gamma^{a(x_j-x_1)}\oplus\gamma^{-ax_j})\cong P(\gamma^{-ax_1}\oplus\uC)\quad\text{as bundles}.
\end{equation}
As remarked at \eqref{eqn:rhoj}, $\rho_i\colon \cH_i\to \cH_{i-1}$ for $i>2$ is a bundle map and pulls back the tautological line bundle over $\cH_{i-1}$ to that over $\cH_i$ so that $\rho_i^*(x_{i-1})=x_i$.  Therefore $\gamma^{x_j}$ is the pullback of $\gamma^{x_2}$ over $\cH_2$ by a composition of the bundle maps $\rho_i$'s.  Moreover $\rho_i^*(x_1)=x_1$ as noted before.  Therefore it suffices to prove \eqref{eqn:(2)} when $j=2$ and the base space is $\cH_2$. But then the two vector bundles $\gamma^{a(x_j-x_1)}\oplus\gamma^{-ax_j}$ and $\gamma^{-ax_1}\oplus\uC$ in \eqref{eqn:(2)} are isomorphic because their total Chern classes are same and complex vector bundles of complex dimension two over $\cH_2$ are classified by their total Chern classes.  In fact the last assertion follows from an exact sequence
$$
    [\cH_2,U/U(2)] \to [\cH_2,BU(2)] \to [\cH_2,BU]=K(\cH_2)
$$
induced from a fibration $U/U(2)\to BU(2)\to BU$.  Here $[\cH_2,U/U(2)]=0$ because $\cH_2$ is of real dimension 4 and $U/U(2)$ is 4-connected
and $K(\cH_2)$ is torsion free since $H^{odd}(\cH_2)=0$, so that elements in $[\cH_2,BU(2)]$ can be distinguished by their Chern classes.
\end{proof}

\section{Cohomological rigidity of $\Q$-trivial Bott manifolds} \label{section: cohomological rigidity of Q-trivial}

For $n \in \N$, a finite sequence $\lambda = (\lambda_1, \ldots, \lambda_m )$ of positive integers is called a \emph{partition} of $n$ if $\sum_{1 \leq i \leq m} \lambda_i = n$ and $\lambda_1 \geq \cdots \geq \lambda_m \geq 1$.
We define $\cH_\lambda$ by
$$
    \cH_\lambda := \cH_{\lambda_1} \times \cdots \times \cH_{\lambda_m}.
$$
For instance, $(\CP^1)^n$ is $\cH_{(1, \ldots, 1)}$ and $\cH_n$ is $\cH_{(n)}$.
Note that
\begin{equation} \label{eqn:tensor}
\text{$H^\ast(\cH_{\lambda})=H^\ast(\cH_{\lambda_1}) \otimes \cdots \otimes H^\ast(\cH_{\lambda_m})$.}
\end{equation}

\begin{theorem} \label{theorem:Q-trivial Bott manifold}
\begin{enumerate}
\item An $n$-stage $\Q$-trivial Bott manifold is diffeomorphic to $\cH_{\lambda}$ for some partition $\lambda$ of $n$.
\item Let $\lambda$ and $\lambda'$ be two partitions of $n$. If $H^\ast(\cH_\lambda)$ is isomorphic to $H^\ast(\cH_{\lambda'})$ as graded rings, then $\lambda = \lambda'$.
\end{enumerate}
\noindent
Therefore, $\Q$-trivial Bott manifolds are distinguished by their cohomology rings with $\Z$-coefficients and the number of diffeomorphism classes in $n$-stage Bott manifolds is equal to the number of partitions of $n$.
\end{theorem}

\begin{proof}
(1) We prove the statement (1) by induction on $n$. Let $B_n$ be an $n$-stage Bott manifold in the tower \eqref{eqn:Bott tower} and suppose that $B_n$ is $\Q$-trivial.    When $n=1$, the statement is trivial  since $B_1=\C P^1=\cH_1$.

Assume the statement (1) holds for $(n-1)$-stage $\Q$-trivial Bott manifolds. Then, since $B_{n-1}$ is also $\Q$-trivial by Proposition~\ref{proposition:Q-trivial Bott manifold},  we may assume that $B_{n-1}=\cH_\mu$ for some partition $\mu$ of $n-1$ by the induction assumption and $B_n = P(\gamma^{\alpha_n} \oplus \uC)$ with $\alpha_n \in H^2(\cH_{\mu})$.  We note that $\alpha_n^2 = 0$ by Proposition~\ref{proposition:Q-trivial Bott manifold} because $B_n$ is $\Q$-trivial.  If $\alpha_n=0$, then $B_n=\cH_\mu\times \cH_1$ and the theorem holds in this case.  Suppose $\alpha_n\not=0$.  Then $\alpha_n$ must sit in $H^2(\cH_{\mu_j})$ for some component $\mu_j$ of the partition $\mu$
in \eqref{eqn:tensor} with $\lambda$ replaced by $\mu$ because otherwise $\alpha_n^2$ cannot vanish.  Therefore the line bundle $\gamma^{\alpha_n}$ over $\cH_\mu$ can be obtained by pulling back  a line bundle over $\cH_{\mu_j}$. It follows that $B_n$ is diffeomorphic to
\[
P(\gamma^{\alpha_n}\oplus\underline{\C})\times \prod_{i\not=j}\cH_{\mu_i}
\]
where $\gamma^{\alpha_n}$ is regarded as a line bundle over $\cH_{\mu_j}$, $\mu_i$ runs over all components of $\mu$ different from $\mu_j$. Then the statement (1) follows from Lemma~\ref{lemma:bundle over H_n}.

(2)  Any (non-zero) square zero element in $H^2(\cH_\lambda)$ sits in $H^2(\cH_{\lambda_i})$ for some component $\lambda_i$ of $\lambda$ as noted above and it follows from Lemma~\ref{lemm:mod2} that the mod 2 reductions of a square zero primitive element in $H^2(\cH_{\lambda_i})$ and that in $H^2(\cH_{\lambda_j})$ are same if and only if $i=j$. Therefore, if $\varphi : H^\ast(\cH_{\lambda}) \to H^\ast(\cH_{\lambda'})$ is a graded ring homomorphism, then all square zero primitive elements in $H^2(\cH_{\lambda_i})$ map into $H^2(\cH_{\lambda'_j})$ by $\varphi$ for some component $\lambda'_j$ of $\lambda'$.  Since the square zero primitive elements in $H^2(\cH_{\lambda_i})$ generate $H^*(\cH_{\lambda_i})$ over $\Q$, this implies that  $\varphi(H^*(\cH_{\lambda_i}))$ is contained in $H^*(\cH_{\lambda'_j})$.  If $\varphi$ is in particular an isomorphism, then this together with \eqref{eqn:tensor} implies the statement (2).
\end{proof}

\begin{remark}
One can show that $\cH_\lambda$'s, in other words $\Q$-trivial Bott manifolds, can be distinguished by their cohomology rings even with $\Z/2$- or $\Z_{(2)}$-coefficients.
It is not true that all Bott manifolds can be distinguished by their cohomology rings with $\Z/2$-coefficients (e.g. 3-stage Bott manifolds are such examples, see \cite{Ch-Ma-Su-2010}), but it might be true with $\Z_{(2)}$-coefficients, see \cite{ch-su-pre}.
\end{remark}

\section{Automorphisms of $\Q$-trivial Bott manifolds} \label{sect:auto}

By Theorem~\ref{theorem:Q-trivial Bott manifold} we may assume that an $n$-stage Bott manifold is $\cH_\lambda$ where $\lambda$ is a partition of $n$. In this section we shall study the group $\Aut(H^*(\cH_\lambda))$ of graded ring automorphisms of $H^*(\cH_\lambda)$ and prove the following.

\begin{theorem} \label{theo:diffeo}
Any element of $\Aut(H^*(\cH_\lambda))$ is induced from a diffeomorphism of $\cH_\lambda$.
\end{theorem}

Since $\Q$-trivial Bott manifolds are distinguished by their cohomology rings by Theorem~\ref{theorem:Q-trivial Bott manifold}, the theorem above implies the following.

\begin{corollary} \label{coro:diffeo}
Any cohomology ring isomorphism between two $\Q$-trivial Bott manifolds is induced from a diffeomorphism.
\end{corollary}

The rest of this section is devoted to the proof of Theorem~\ref{theo:diffeo}.  Remember that the square zero primitive elements in $H^2(\cH_n)$ are $\pm x_1$ and $\pm(2x_j-x_1)$ for $j>1$ by Lemma~\ref{lemm:mod2}.

\begin{lemma}
An automorphism of $H^*(\cH_n)$ permutes  $\pm x_1$ and $\pm(2x_j-x_1)$ for $j>1$ up to sign.  On the other hand, any permutation of $\pm x_1$ and $\pm(2x_j-x_1)$ for $j>1$ up to sign induces an automorphism of $H^*(\cH_n)$.

Therefore, $\Aut(H^*(\cH_n))$ is isomorphic to a semi-direct product $(\Z/2)^n\rtimes \frak S_n$ where $\frak S_n$ denotes the symmetric group on $n$ letters and the action of $\frak S_n$ on $(\Z/2)^n$ is the natural permutation of factors of $(\Z/2)^n$.
\end{lemma}

\begin{proof}
The first statement is obvious.  Suppose that $\varphi$ is a permutation of $\pm x_1$ and $\pm(2x_j-x_1)$ for $j>1$ up to sign.  Then $\varphi(x_1)=\pm x_1$ or $\pm(2x_k-x_1)$ for some $k>1$.  In any case one can easily check that if we extend $\varphi$ linearly, then $\varphi(x_i)$ is integral (i.e., a linear combination of $x_\ell$'s over $\Z$) for any $i$.  For instance, if
\[
\varphi(x_1)=2x_k-x_1,\ \varphi(2x_i-x_1)=x_1,\ \varphi(2x_j-x_1)=-(2x_\ell-x_1) \text{ for $j\not=i$},
\]
then a simple computation shows that
\[
\varphi(x_i)=x_k\text{ and } \varphi(x_j)=x_k-x_\ell.
\]
Thus the linear extension of $\varphi$ defines an endomorphism of $H^2(\cH_n)$.  Moreover, one can also check that $\varphi(x_1)^2=0$ and $\varphi(x_j)^2=\varphi(x_1)\varphi(x_j)$ for $j>1$.  This ensures that $\varphi$ extends to a graded ring endmorphism $\overline{\varphi}$ of $H^*(\cH_n)$ since the ideal in \eqref{eqn:Hn} is generated by $x_1^2$ and $x_j^2-x_1x_j$ for $j>1$.  Similarly, $\varphi^{-1}$ induces a graded ring endomorphism $\overline{\varphi^{-1}}$ of $H^*(\cH_n)$ and clealry $\overline{\varphi^{-1}}$ gives the inverse of $\overline{\varphi}$, so $\overline{\varphi}$ is an automorphism of $H^*(\cH_n)$.  This proves the lemma.
\end{proof}

We write $\lambda=(d_1^{a_1},\dots,d_k^{a_k})$ where $d_1>\dots>d_k$ and $d_i^{a_i}$ denotes $a_i$ copies of $d_i$ for $i=1,\dots,k$. Then
\[
H^*(\cH_\lambda)=\bigotimes_{i=1}^k H^*(\cH_{d_i})^{\otimes a_i}.
\]
The proof of (2) in Theorem~\ref{theorem:Q-trivial Bott manifold} shows that an automorphism of $H^*(\cH_\lambda)$ maps factors of $H^*(\cH_{d_i})^{\otimes a_i}$ to themselves for each $i$, so that
\begin{equation} \label{eqn:AutH}
\Aut(H^*(\cH_\lambda))=\prod_{i=1}^k\Aut(H^*(\cH_{d_i})^{\otimes a_i})=\prod_{i=1}^k\Aut(H^*(\cH_{d_i}))^{a_i}\rtimes \frak S_{a_i}
\end{equation}
where the action of $\frak S_{a_i}$ on $\Aut(H^*(\cH_{d_i}))^{a_i}$ is the natural permutation of factors of $\Aut(H^*(\cH_{d_i}))^{a_i}$.

A permutation of factors of $\Aut(H^*(\cH_{d_i}))^{a_i}$ is induced from a permutation of factors of $\cH_{d_i}^{a_i}$, which is a diffeomorphism, so it suffices to prove Theorem~\ref{theo:diffeo} when $\lambda=(n)$ by \eqref{eqn:AutH}.  We first prove it when $n=2$.

\begin{lemma} \label{lemm:H2}
Any element of $\Aut(H^*(\cH_2))$, which permutes $\pm x_1$ and $\pm(2x_2-x_1)$ up to sign, is induced from a diffeomorphism of $\cH_2$.
\end{lemma}

\begin{proof}
As remarked in Example~\ref{example:Hirzebruch surface}, $\cH_2=\Sigma_1$ is diffeomorphic to $\CP^2\#\overline{\CP^2}$.  Let $u$ and $v$ be elements of $H_2( \CP^2\#\overline{\CP^2})$ represented by a canonical submanifold $\CP^1$ in $\CP^2$ and $\overline{\CP^2}$ respectively.  They are a basis of $H_2( \CP^2\#\overline{\CP^2})$.  (Through the Poincar\'e duality, $u$ and $v$ correspond to $x_2$ and $x_1-x_2$ up to sign since the self-intersection numbers of $u$ and $v$ are $\pm 1$ while squares of $x_2$ and $x_2-x_1$ are a cofundamental class $x_1x_2$ up to sign.) It suffices to show that any permutation of $\pm u$ and $\pm v$ up to sign can be represented by a diffeomorphism of $\CP^2\#\overline{\CP^2}=\cH_2$ since the number of those permutations is 8 which agrees with the number of elements in $\Aut(H^*(\cH_2))\cong (\Z/2)^2\rtimes \frak S_2$.

We consider two involutions $s$ and $t$ on $\CP^2$ defined by
\[
s\colon [z_1,z_2,z_3]\to [\bar{z}_1,\bar{z}_2,\bar{z}_3],\qquad t\colon  [z_1,z_2,z_3]\to [z_1,z_2,-z_3]
\]
where $[z_1,z_2,z_3]$ denotes the homogenous coordinate of $\CP^2$ and $\bar{z}$ denotes the complex conjugate of a complex number $z$. Observe that
\begin{enumerate}
\item $s$ leaves the submanifold $\CP^1=\{z_3=0\}$ of $\CP^2$ invariant, reverses an orientation on the $\CP^1$ and the fixed point set of $s$ is $\R P^2$,
\item the induced action of $t$ on $H_\ast(\CP^2)$ is trivial and the fixed point set of $t$ is the disjoint union of $\CP^1=\{z_3=0\}$ and a point $[0,0,1]$.
\end{enumerate}

\noindent
{\bf Type 1.} We consider the involution $s$ on both $\CP^2$ and $\overline{\CP^2}$.  Choose a point from the fixed set $\R P^2$ in $\CP^2$ and $\overline{\CP^2}$ respectively and take equivariant connected sum of $\CP^2$ and $\overline{\CP^2}$ around the chosen points.  Then the resulting involution on $\CP^2\#\overline{\CP^2}$ sends $(u,v)$ to $(-u,-v)$.

\noindent
{\bf Type 2.} We consider the involution $s$ on $\CP^2$ and $t$ on $\overline{\CP^2}$. Choose a point from the fixed set $\R P^2$ in $\CP^2$ and a point from the fixed set $\CP^1$ in  $\overline{\CP^2}$ and take equivariant connected sum of $\CP^2$ and $\overline{\CP^2}$ around the chosen points.  Then the resulting involution on $\CP^2\#\overline{\CP^2}$ sends $(u,v)$ to $(-u,v)$.

\noindent
{\bf Type 3.}  $\CP^2\#\overline{\CP^2}$ is obtained by removing an open disk $D$ from $\CP^2$ and $\overline{\CP^2}$ respectively and gluing together along the boundary $S^3$ via the identity map, so that it admits a reflection with respect to the $S^3$, which maps $\CP^2\backslash D$ to $\overline{\CP^2}\backslash D$.  This reflection sends $(u,v)$ to $(v,u)$.

Combining the diffeomorphisms of the three types above, one can realize any element of $\Aut(H^*(\cH_2))$ by a diffeomorphism of $\cH_2$.
\end{proof}

We shall prove that any element of $\Aut(H^*(\cH_n))$ is induced from a diffeomorphism of $\cH_n$ for any $n$ by induction on $n$, so that the proof of Theorem~\ref{theo:diffeo} will be completed.  For that we prepare three lemmas.  We regard $H^*(\cH_{j})$ for $j<n$ as a subring of $H^*(\cH_n)$ as usual and remember that  $\pm x_1$ and $\pm(2x_j-2x_1)$ for $j>1$ are all the square zero primitive elements in $H^2(\cH_n)$.

\begin{lemma} \label{lemm:ext}
Let $\psi$ be an element of $\Aut(H^*(\cH_{j}))$ for $j<n$.  If $\psi$ is induced from a diffeomorphism of $\cH_{j}$, then there is a diffeomorphism of $\cH_n$ whose induced automorphism of $H^*(\cH_n)$ preserves the subring $H^*(\cH_{j})$ and agrees with the given $\psi$ on $H^*(\cH_j)$.
\end{lemma}

\begin{proof}
Let $f_j$ be a diffeomorphism of $\cH_j$ whose induced automorphism of $H^*(\cH_j)$ is $\psi$.  The pullback of the bundle
\begin{equation} \label{eqn:bundle}
\text{$\cH_{j+1}=P(\uC\oplus \gamma^{\alpha_{j+1}})\stackrel{\pi_{j+1}}\longrightarrow  \cH_j$}
\end{equation}
by $f_j$ is of the form $P(\uC\oplus\gamma^{f_j^*(\alpha_{j+1})})\to \cH_j$ but this is isomorphic to \eqref{eqn:bundle} by Lemma~\ref{lemma:bundle over H_n} since $\alpha_{j+1}^2=0=f_j^*(\alpha_{j+1})^2$ and the mod 2 reductions of $\alpha_{j+1}$ and $f_j^*(\alpha_{j+1})$ are same.  It follows that there is a bundle automorphism $f_{j+1}$ of \eqref{eqn:bundle} which covers $f_j$.  Since $f_{j+1}$ covers $f_j$, the automorphism $f_{j+1}^*$ of $H^*(\cH_{j+1})$ induced by $f_{j+1}$ preserves the subring $H^*(\cH_j)$ and agrees with $f_j^*$ on it.  Repeating this argument for $f_{j+1}$ in place of $f_j$, we get a diffeomorphism $f_{j+2}$ of $\cH_{j+2}$ which covers $f_{j+1}$ and so on.  Then the last diffeomorphism $f_n$ of $\cH_n$ is the desired one.
\end{proof}

\begin{lemma} \label{lemm:xn}
There is a diffeomorphism of $\cH_n$ whose induced automorphism of $H^*(\cH_n)$ is the identity on the subring $H^*(\cH_{n-1})$ and
maps $x_n$ to $-x_n+x_1$ (equivalently maps $2x_n-x_1$ to $-(2x_n-x_1)$).
\end{lemma}

\begin{proof}
Since the dual bundle of $\uC\oplus\gamma^{x_1}$ is isomorphic to $\uC\oplus\gamma^{-x_1}$, the proof of Lemma~\ref{lemm:line} shows that we have a bundle map
\[
\text{$b\colon \cH_n=P(\uC\oplus \gamma^{x_1})\to P(\uC\oplus\gamma^{-x_1})$}
\]
which covers the identity map on $\cH_{n-1}$. The pullback of the tautological line bundle $\eta_-$ over $P(\uC\oplus\gamma^{-x_1})$ by $b$ is complex conjugate to the tautological line bundle $\eta_+$ over $P(\uC\oplus\gamma^{x_1})$ (see Remark~\ref{rema:line}); so we obtain
\begin{equation} \label{eqn:c1}
b^*(x)=-x_n
\end{equation}
where $x=c_1(\eta_-)$ and $x_n=c_1(\eta_+)$ by the definition of $x_n$.

On the other hand, the proof of Lemma~\ref{lemm:line} shows that we have a bundle isomorphism
\[
c\colon P(\uC\oplus\gamma^{-x_1})\to P((\uC\oplus\gamma^{-x_1})\otimes\gamma^{x_1})=P(\gamma^{x_1}\oplus\uC)=\cH_n
\]
which preserves the complex structures on each fiber.  Therefore it induces a \emph{complex} vector bundle isomorphism $T_fP(\uC\oplus\gamma^{-x_1})\to T_fP(\gamma^{x_1}\oplus\uC)$ between their tangent bundles along the fibers.  According to the Borel-Hirzebruch formula \eqref{eqn:tf}, their first Chern classes are respectively $-2x-x_1$ and $-2x_n+x_1$, so
\begin{equation} \label{eqn:c2}
c^*(-2x_n+x_1)=-2x-x_1.
\end{equation}
Since the map $c$ covers the identity map on $\cH_{n-1}$, $c^*(x_1)=x_1$.  It follows from \eqref{eqn:c2} that $c^*(x_n)=x+x_1$.  This together with \eqref{eqn:c1} shows that
\begin{equation} \label{eqn:bc}
b^*(c^*(x_n))=-x_n+x_1
\end{equation}
because $b^*(x_1)=x_1$ which follows from the fact that $b$ covers the identity map on $\cH_{n-1}$.
The identity \eqref{eqn:bc} shows that the composition $c\circ b$ is the desired diffeomorphism.
\end{proof}

\begin{lemma} \label{lemm:xnxj}
There is a diffeomorphism of $\cH_n$ whose induced automorphism of $H^*(\cH_n)$ interchanges $x_i$ and $x_j$ for $i,j>1$ and fixes $x_k$ for $k\not=i,j$.
\end{lemma}

\begin{proof}
It suffices to show that there is a diffeomorphism $g_i$ of $\cH_n$ for each $i>1$ whose induced automorphism of $H^*(\cH_n)$ interchanges $x_i$ and $x_{i+1}$ and fixes $x_k$ for $k\not=i,i+1$, because the desired diffeomorphism can be obtained by composing those diffeomorphisms.

Remember that $\cH_{i+1}$ is obtained as the fiber product
\[
\begin{CD}
\cH_{i+1} @>\rho_{i+1}>> \cH_i\\
@VV\pi_{i+1}V   @VV\pi_iV\\
\cH_i @>\pi_i>> \cH_{i-1}.
\end{CD}
\]
Permuting the coordinates of $\cH_i\times \cH_i$ preserves the subset $\cH_{i+1}$ and defines a diffeomorphism $\tau_{i+1}$ of $\cH_{i+1}$.  One notes that  $\tau_{i+1}^*(x_i)=\rho_{i+1}^*(x_i)=x_{i+1}$ and $\tau_{i+1}^*(x_k)=x_k$ for $k<i$.  Since $\pi_{i+1}\circ \tau_{i+1}=\pi_{i+1}$, the diffeomorphism $\tau_{i+1}$ naturally extends to a diffeomorphism $\tau_{i+2}$ of $\cH_{i+2}$ and finally extends to a diffeomorphism $g_i$ of $\cH_n$ because of \eqref{eqn:suyoung_tower}.  Since $\tau_{i+1}^*(x_1)=x_1$, the pullback of the line bundle $\gamma^{x_1}$ over $\cH_{i+1}$ is isomorphic to $\gamma^{x_1}$ itself.  This implies that $\tau_{i+2}^*(x_{i+2})=x_{i+2}$ because $x_{i+2}$ is the first Chern class of the tautological line bundle over $P(\uC\oplus \gamma^{x_1})$.  Therefore $g_i^*$ fixes $x_{i+2}$ since $g_i$ is an extension of $\tau_{i+2}$.  Similarly, $g_i^*$ fixes $x_k$ for $k>i+1$.  Thus $g_i$ is the desired diffeomorphism.
\end{proof}

\begin{remark}
As remarked at \eqref{eqn:quotient}, one can regard $\cH_n$ as the quotient of $(S^3)^n$ by a free action of $(S^1)^n$ associated with the matrix \eqref{eqn:matrixHn}.  Then interchanging the $i$-th factor and the $j$-th factor of $(S^3)^n$ produces a desired diffeomorphism in Lemma~\ref{lemm:xnxj}.
\end{remark}

Now we shall prove that any element of $\Aut(H^*(\cH_n))$ is induced from a diffeomorphism of $\cH_n$ for any $n$ by induction on $n$.  This claim is established for $n=2$ by Lemma~\ref{lemm:H2}.  Suppose the claim holds for $n-1$.  Let $\varphi$ be an element of $\Aut(H^*(\cH_n))$.  Then $\varphi$ permutes square zero primitive elements $\pm x_1, \pm(2x_j-x_1)$ $(j>1)$ up to sign.  We distinguish three cases.

{\bf Case 1.}  The case where $\varphi(2x_n-x_1)=\pm(2x_n-x_1)$.  In this case $\varphi$ preserves the subring $H^*(\cH_{n-1})$ and let $\psi$ be the restriction of $\varphi$ to $H^*(\cH_{n-1})$.  By Lemma~\ref{lemm:ext} there is a diffeomorphism $f$ of $\cH_n$ whose induced automorphism $f^*$ of $H^*(\cH_n)$ agrees with $\psi$ on $H^*(\cH_{n-1})$.  Then the composition $(f^{-1})^*\circ \varphi$ is the identity on $H^*(\cH_{n-1})$, so we may assume that $\varphi$ is the identity on $H^*(\cH_{n-1})$.  If $\varphi(2x_n-x_1)=2x_n-x_1$, then $\varphi$ is the identity so that it is induced from the identity diffeomorphism of $\cH_n$.  If $\varphi(2x_n-x_1)=-(2x_n-x_1)$, then  $\varphi$ is induced from a diffeomorphism of $\cH_n$ by Lemma~\ref{lemm:xn}.

{\bf Case 2.}  The case where $\varphi(2x_n-x_1)=\pm(2x_j-x_1)$ for some $1<j<n$.   By Lemma~\ref{lemm:xnxj} there is a diffeomorphism $g$ of $\cH_n$ whose induced automorphism $g^*$ of $H^*(\cH_n)$ interchanges $x_j$ and $x_n$ and fixes $x_k$ for $k\not=j,n$.  Therefore the composition $g^*\circ \varphi$ is an automorphism treated in Case 1, so that $g^*\circ \varphi$ is induced from a diffeomorphism of $\cH_n$ by Case 1 and hence so is $\varphi$.

{\bf Case 3.}  The case where $\varphi(2x_n-x_1)=\pm x_1$.  By Lemma~\ref{lemm:H2} and Lemma~\ref{lemm:ext}, there is a diffeomorphism $h$ of $\cH_n$ whose induced automorphism $h^*$ of $H^*(\cH_n)$ maps $x_1$ to $2x_2-x_1$.  Therefore the composition $h^*\circ \varphi$ is an automorphism treated in Case 2, so that it is induced from a diffeomorphism of $\cH_n$ and hence so is $\varphi$.

This completes the proof of the desired claim and hence Theorem~\ref{theo:diffeo}.

\medskip
{\bf Concluding remark.}
The cohomological rigidity problem asks whether two toric manifolds are diffeomorphic (or homeomorphic) if their cohomology rings are isomorphic. More strongly, it is asked in \cite{Ma-Su-2008} whether any cohomology ring isomorphism between two toric manifolds is induced from a diffeomorphism.  We may call this problem the \emph{strong cohomological rigidity problem} for toric manifolds.  Corollary~\ref{coro:diffeo} gives a supporting evidence to the problem and the authors do not know any counterexample to the problem.

\bigskip
\bibliographystyle{amsplain}

\end{document}